\documentclass{amsproc}

\usepackage{graphicx}

\usepackage{}

\newtheorem{theorem}{Theorem}[section]
\newtheorem{lemma}[theorem]{Lemma}

\theoremstyle{definition}
\newtheorem{definition}[theorem]{Definition}

\newtheorem{proposition}[theorem]{Proposition}

\newtheorem{claim}[theorem]{Claim}

\theoremstyle{remark}

\numberwithin{equation}{section}

\begin{document}

\title{Stable equivalence of bridge positions of a handlebody-knot}


\author{Makoto Ozawa}
\address{Department of Natural Sciences, Faculty of Arts and Sciences, Komazawa University, 1-23-1 Komazawa, Setagaya-ku, Tokyo, 154-8525, Japan}
\curraddr{}
\email{w3c@komazawa-u.ac.jp}
\thanks{The author is partially supported by Grant-in-Aid for Scientific Research (C) (No. 17K05262) and (B) (No. 16H03928), The Ministry of Education, Culture, Sports, Science and Technology, Japan}



\date{}

\begin{abstract}
We show that any two bridge positions of a handlebody-knot are stably equivalent.
\end{abstract}

\maketitle


\section{Introduction}

Reidemeister (\cite{Re}) and Singer (\cite{Si}) independently proved the stable equivalence theorem, that is, for any two Heegaard splittings $H_1$ and $H_2$ of a 3-manifold $M$, there exists a third Heegaard splitting $H$ which is a stabilization of both.
There are another proofs in \cite{C}, \cite{Sie}.
A modern, simple proof of the Reidemeister--Singer Theorem is given by Lei (\cite{L}).
Those proofs are worked in the piecewise linear category.
On the other hand, in the smooth category, Johnson (\cite{J}) gives a proof of Reidemeister--Singer Theorem using Rubinstein and Scharlemann's graphic, and Laudenbach (\cite{La}) gives a proof of Reidemeister--Singer Theorem by Cerf's methods.

In the context of knot theory, bridge positions (or bridge decompositions) of a knot correspond to Heegaard splittings of a 3-manifold.
Birman (\cite{B}) proved that the stable equivalence of plat representations, namely, bridge positions.
Hayashi (\cite{H}) proved that the stable equivalence theorem for a pair of (3-manifold, 1-submanifold), which generalizes both of Reidemeister--Singer Theorem and Birman's theorem.
Along the proof by Lei (\cite{L}), Zupan (\cite{Z}) gives a short proof of Hayashi's theorem.

Ishii (\cite{I}) opened a road to handlebody-knots, that is, handlebodies embedded in the 3-sphere.
He classified the moves on diagrams of trivalent graphs embedded in the 3-sphere, which define the same handlebody-knot ({\cite[Corollary 2]{I}}).
See Theorem \ref{R} for the details.
In this paper, we will prove the stable equivalence of bridge positions of a handlebody-knot (Theorem \ref{stable1}), and equivalently, the stable equivalence of bridge positions of a trivalent graph (Theorem \ref{stable2}).
To prove Theorem \ref{stable1} and \ref{stable1}, we introduce a regular bridge position of a knotted trivalent graph.
We will show that any bridge position can be isotoped by a horizontal isotopy so that it is regular (Lemma \ref{compatible}).
Then we can handle regular diagrams and bridge positions simultaneously, and by virtue of Reidemeister moves for trivalent graphs (\cite[Corollary 2]{I}), we will show that stable equivalence of regualr bridge positions (Theorem \ref{stable3}).

In the proof of the stable equivalence of bridge positions of a trivalent graph, the second stabilization move $S2$ (Figure \ref{S123}) is the key to a solution.
As far as the author knows, the stabilization $S2$ was first appeared in \cite{D} as moves 3, 4, 5, 6.
Dancso probably has known all moves which relate all Morse positions of a trivalent graph embedded in the 3-sphere, but the proof was not included in the paper \cite{D}.
Ishihara--Ishii (\cite{II}) defined the sliced diagram of trivalent tangles, that is, a diagram in Morse position, and showed that two sliced diagrams are related by an isotopy of the plane if and only if they are related by a finite sequence of moves $Pi$, $Pii$, $Piii$, $Piv\ (=S2)$ ({\cite[Theorem 5.1 (1)]{II}}).
This refines Dancso's unproved theorem in the context of diagrams in Morse position.
Dancso's unproved theorem or Ishihara--Ishii's theorem overlaps our main result Theorem \ref{stable2}, however it does not imply Theorem \ref{stable2} since for given two bridge positions of two trivalent graphs, we will construct two sequences of stabilizations and moves from two bridge positions to a third bridge position.
Thus it is irreversible.
By Theorem \ref{stable2}, we also give a proof of Dancso's unproved theorem in Proposition \ref{related}.

\setcounter{tocdepth}{2}
\tableofcontents

\section{Definitions}

\subsection{Height function and projection}

Recall that $S^3=\{(x,y,z,w)\in \Bbb{R}^4 \mid x^2+y^2+z^2+w^2=1 \}$ is decomposed by $S^2=\{(x,y,z,0)\in \Bbb{R}^4 \mid x^2+y^2+z^2=1 \}$ into two 3-balls $B_+=\{(x,y,z,w)\in \Bbb{R}^4 \mid x^2+y^2+z^2+w^2=1, w\ge0 \}$ and $B_-=\{(x,y,z,w)\in \Bbb{R}^4 \mid x^2+y^2+z^2+w^2=1, w\le0 \}$.
We call two points $(0,0,0,1)$ and $(0,0,0,-1)$ the {\em north pole} and {\em south pole} of $S^3$, and denote by $+\infty$ and $-\infty$ respectively.

Let $h:\Bbb{R}^4\to \Bbb{R}$ be a height function defined by $(x,y,z,w)\mapsto(0,0,0,w)$.
We denote the restriction $h|_{S^3}$ by the same symbol $h$.

Let $r:\Bbb{R}^4-(\mbox{w-axis})\to S^2\times \Bbb{R}$ be a retraction defined by 
$$\displaystyle (x,y,z,w)\mapsto \Big(\frac{x}{\sqrt{x^2+y^2+z^2}}, \frac{y}{\sqrt{x^2+y^2+z^2}}, \frac{z}{\sqrt{x^2+y^2+z^2}}, w\Big).$$
Then the restriction $r|_{S^3-\{\pm\infty\}}$ defines a map $S^3-\{\pm\infty\}\to S^2\times(-1,1)$ and denote it by the same symbol $r$.
Let $\pi:S^2\times (-1,1)\to S^2$ be a projection defined by $(x,y,z,w)\mapsto (x,y,z,0)$.
Then the composition $\pi\circ r$ defines a projection $S^3-\{\pm\infty\}\to S^2$ and denote it by $p$.

\subsection{Essential saddles, $\lambda$-vertices and $Y$-vertices}


Let $F$ be a closed surface embedded in $S^3$ and suppose that $h|_F$ is a Morse function and all critical
points have distinct critical values.
Let $x$ be a saddle point of $F$ which corresponds to the critical value $t_x\in \Bbb{R}$.
Let $P_x$ be a pair of pants component of $F\cap h^{-1}([t_x-\epsilon, t_x+\epsilon])$ containing $x$ for a sufficiently small positive real number $\epsilon$.
Let $C_x^1$, $C_x^2$ and $C_x^3$ be the boundary components of $P_x$, where we assume that $C_x^1$ and $C_x^2$ are contained in the same level $h^{-1}(t_x\pm\epsilon)$, and $C_x^3$ is contained in the another level $h^{-1}(t_x\mp\epsilon)$.
A saddle point $x$ of $F$ is {\em upper} (resp. {\em lower}) if $C_x^1$ and $C_x^2$ are contained in $h^{-1}(t_x-\epsilon)$ (resp. $h^{-1}(t_x+\epsilon)$).
A saddle point $x$ of $F$ is {\em essential} if both of the two loops $C_x^1$ and $C_x^2$ are essential in $F$, and it is {\em inessential} if it is not essential.



Let $\Gamma$ be a trivalent graph in $S^3$.
A vertex $v$ of $\Gamma$ is called a {\em $\lambda$-vertex} (resp. {\em $Y$-vertex}) with respect to the height function $h$ if two ends of incident edges lie below $v$ (resp. above $v$).

\subsection{Morse positions for handlebody-knots}

Let $v, v'\subset S^3$ be two handlebodies embedded in the 3-sphere.
We say that $v, v'$ are {\em equivalent} if there exists an ambient isotopy $\{f_t\}_{t\in [0,1]}$ of $S^3$ taking $v$ to $v'$.

Let $V$ be a handlebody-knot type (i.e. an equivalence class of handlebody-knots).
We always assume that $v\in V$ does not contain the north and south poles $\pm\infty=(0,0,0,\pm1)$ of $S^3$.

\begin{definition}
We say that $v\in V$ is a {\em Morse position} if
\begin{enumerate}
\item $h|_{\partial v}$ is a Morse function.
\item All critical points of $h|_{\partial v}$ have distinct values.
\item For any regular value $t$ of $h|_{\partial v}$, $v$ intersects $h^{-1}(t)$ only in disks.
\item Each saddle point of $h|_{\partial v}$ is essential.
\end{enumerate}
\end{definition}

\begin{definition}
Two Morse positions $v, v'\in V$ are {\em equivalent} if there exists an ambient isotopy $\{f_t\}_{t\in [0,1]}$ of $S^3$ taking $v$ to $v'$ such that 
\begin{enumerate}
\item For any $t\in [0,1]$, $h|_{f_t(\partial v)}$ is a Morse function.
\item For any $t\in[0,1]$, there is no pair of mimimum/lower saddle and maximum/upper saddle of $h|_{f_t(\partial v)}$ with the same critical value.
\end{enumerate}
We call such an ambient isotopy $\{f_t\}_{t\in [0,1]}$ a {\em Morse isotopy}.
\end{definition}

\subsection{Morse positions for knotted trivalent graphs}

Let $\gamma, \gamma'\subset S^3$ be two trivalent graphs embedded in the 3-sphere.
We say that $\gamma, \gamma'$ are {\em equivalent} if there exists an ambient isotopy $\{f_t\}_{t\in [0,1]}$ of $S^3$ taking $\gamma$ to $\gamma'$.

Let $\Gamma$ be a knotted trivalent graph type (i.e. an equivalence class of knotted trivalent graphs).
We always assume that $\gamma\in \Gamma$ does not intersect the north and south poles $\pm\infty=(0,0,0,\pm1)$ of $S^3$.

\begin{definition}[Normal form of \cite{GST}, \cite{S}, \cite{ST}]
We say that $\gamma\in \Gamma$ is a {\em Morse position} if
\begin{enumerate}
\item For each edge $e$, the critical points of $h|_e$ are nondegenerate and each lies in the interior of $e$.
\item All critical points of $h|_{\gamma}$ and vertices have distinct values.
\item Each vertex is either $Y$-vertex or $\lambda$-vertex.
\end{enumerate}
\end{definition}

\begin{definition}
Two Morse positions $\gamma, \gamma'\in \Gamma$ are {\em equivalent} if there exists an ambient isotopy $\{f_t\}_{t\in [0,1]}$ of $S^3$ taking $\gamma$ to $\gamma'$ such that 
\begin{enumerate}
\item For any $t\in [0,1]$ and for each edge $e$, the critical points of $h|_{f_t(e)}$ are nondegenerate and each lies in the interior of $e$.
\item For any $t\in [0,1]$, each vertex $f_t(v)$ is either $Y$-vertex or $\lambda$-vertex.
\item For any $t\in[0,1]$, there is no pair of mimimum/$Y$-vertex and maximum/$\lambda$-vertex of $h|_{f_t(\gamma)}$ with the same critical value.
\end{enumerate}
We call such an ambient isotopy $\{f_t\}_{t\in [0,1]}$ a {\em Morse isotopy}.
\end{definition}

\subsection{Several moves for Morse positions}

We prepare several moves for Morse positions of knotted trivalent graphs.
By taking the neighborhood, these moves can be considered as moves for Morse positions of handlebody-knots.

First we define $B1$, $B2$, $B3$ moves as shown in Figure \ref{B123}.
These moves are reversible and obtained by Morse isotopies for both knotted trivalent graphs and handlebody-knots.

\begin{figure}[htbp]
	\begin{center}
	\begin{tabular}{cc}
	\includegraphics[width=0.35\textwidth,pagebox=cropbox,clip]{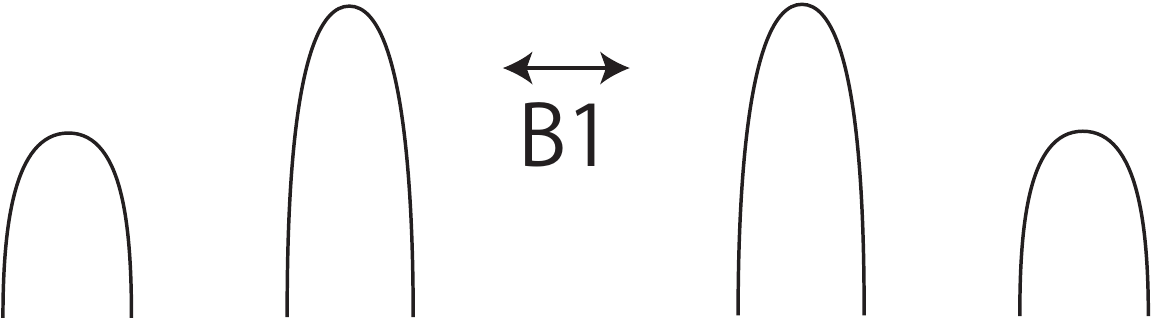}&
	\includegraphics[width=0.35\textwidth,pagebox=cropbox,clip]{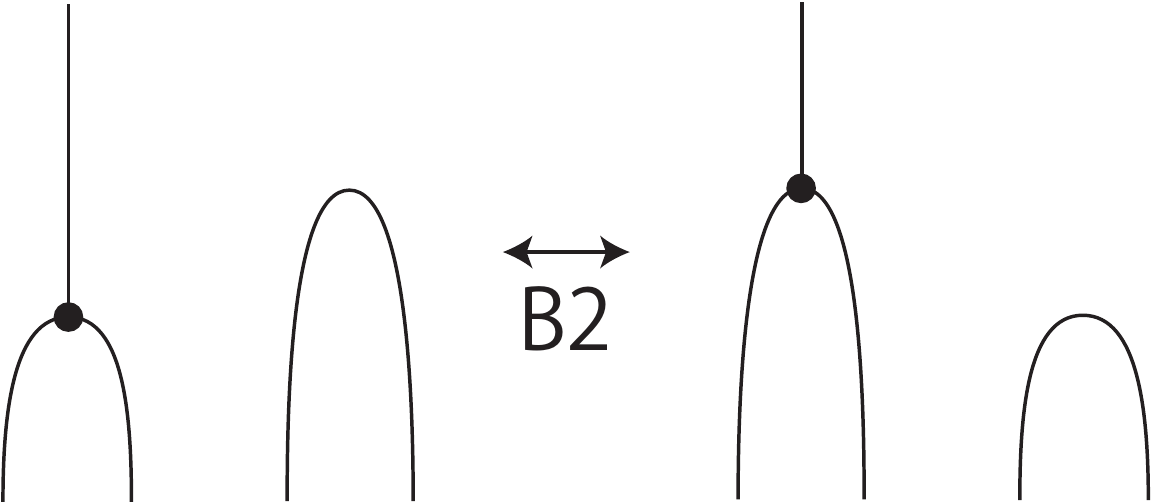}\\
	\includegraphics[width=0.35\textwidth,pagebox=cropbox,clip]{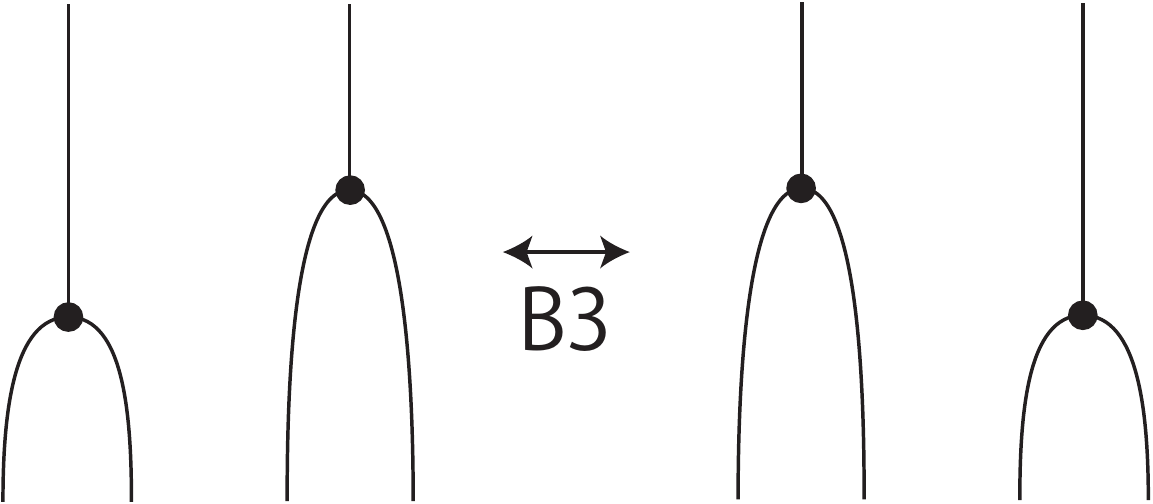}&
	\end{tabular}
	\end{center}
	\caption{$B1$, $B2$, $B3$ moves}
	\label{B123}
\end{figure}

Next we define $B4$, $B5$ moves as shown in Figure \ref{B45}.
These moves are reversible and obtained by Morse isotopies for only handlebody-knots.
Those are not Morse isotopies for knotted trivalent graphs.
The $B4$ move is appeared as 7, 8 moves in \cite{D}.
The $B5$ move is appeared as $Rvi$ in \cite{II}.

\begin{figure}[htbp]
	\begin{center}
	\begin{tabular}{cc}
	\includegraphics[width=0.35\textwidth,pagebox=cropbox,clip]{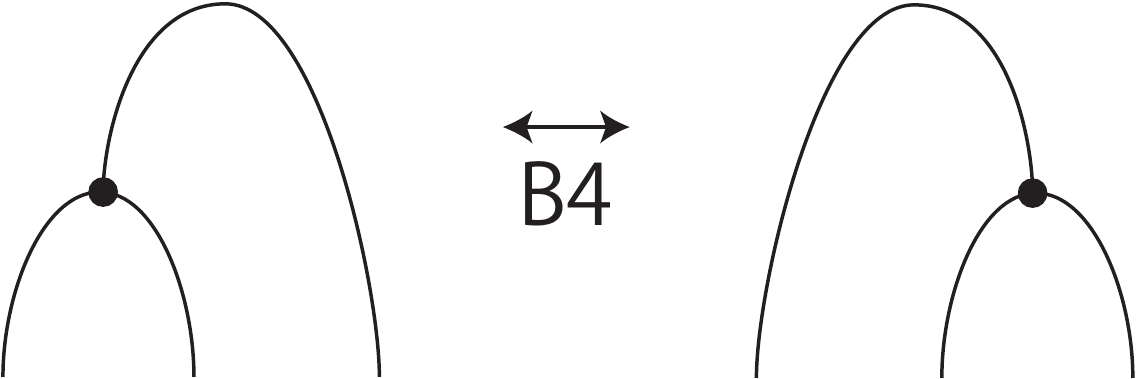}&
	\includegraphics[width=0.35\textwidth,pagebox=cropbox,clip]{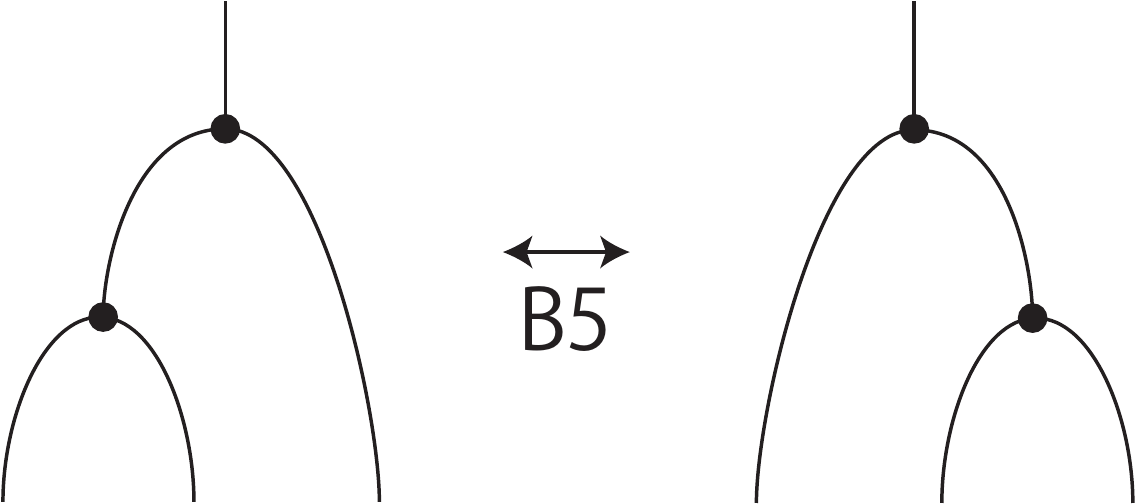}
	\end{tabular}
	\end{center}
	\caption{$B4$, $B5$ moves}
	\label{B45}
\end{figure}

Next we define stabilizations $S1$, $S2$, $S3$ as shown in Figure \ref{S123}.
These stabilizations are irreversible and not obtained by Morse isotopies for both knotted trivalent graphs and handlebody-knots.
The reverse move of a stabilization is called a {\em destabilization}.
The stabilization $S1$ and its destabilization are appeared as 1, 2 moves in \cite{D} and $Pii$ in \cite{II}.
The stabilization $S2$ and its destabilization are appeared as 3, 4, 5, 6 moves in \cite{D} and $Piv$ in \cite{II}.

\begin{figure}[htbp]
	\begin{center}
	\begin{tabular}{cc}
	\includegraphics[width=0.35\textwidth,pagebox=cropbox,clip]{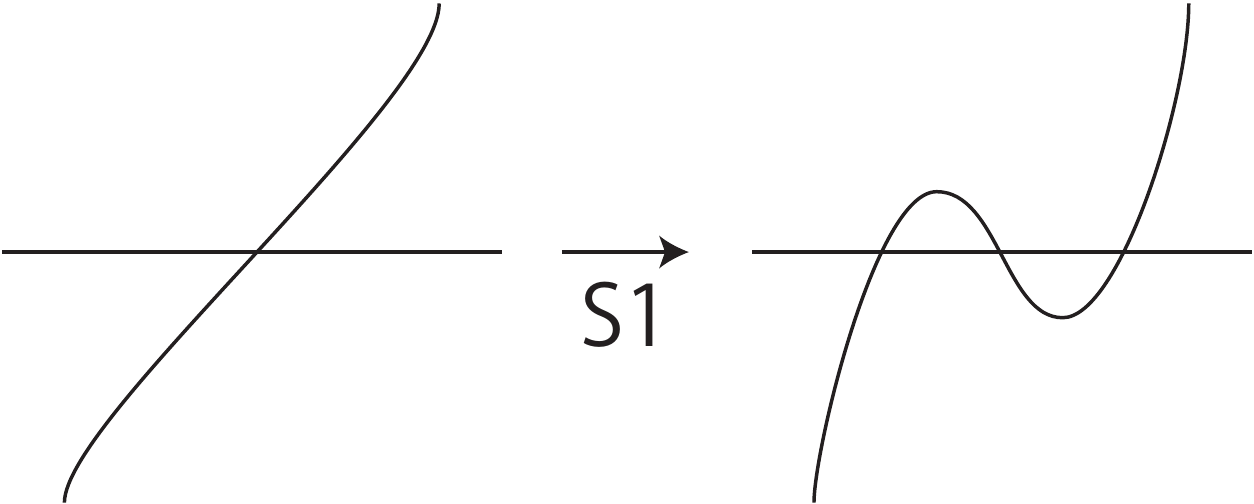}&
	\includegraphics[width=0.35\textwidth,pagebox=cropbox,clip]{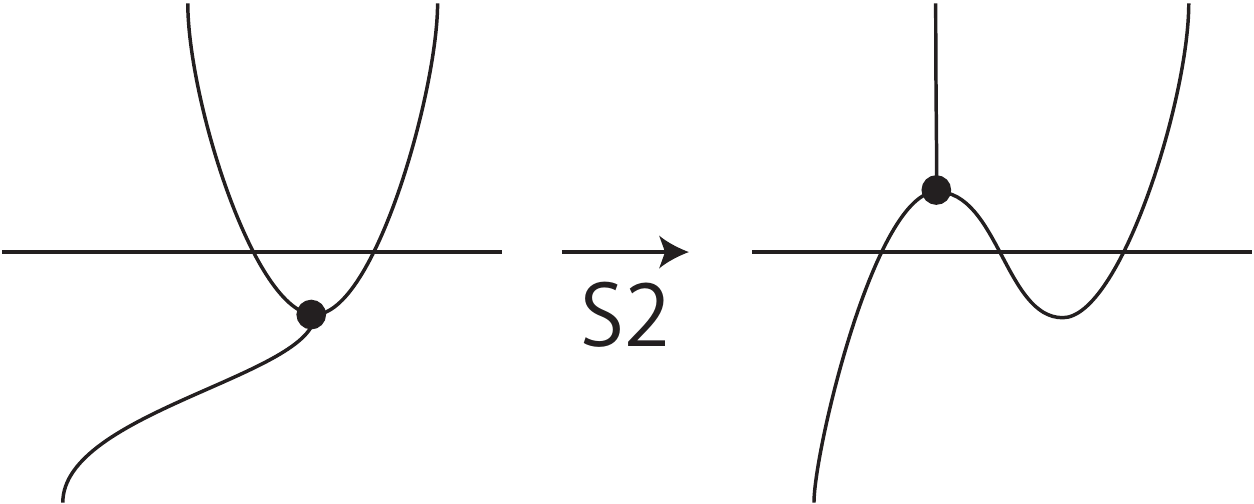}\\
	\includegraphics[width=0.35\textwidth,pagebox=cropbox,clip]{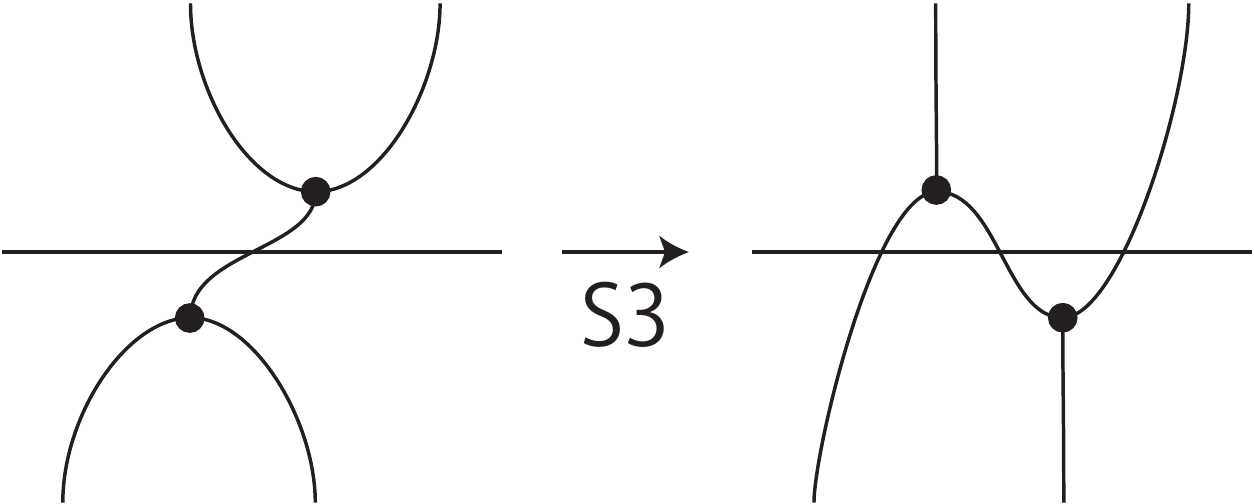}&
	\end{tabular}
	\end{center}
	\caption{Stabilizations $S1$, $S2$, $S3$}
	\label{S123}
\end{figure}

Finally we define $M1$, $M2$, $M3$ moves as shown in Figure \ref{M123}.
These moves are irreversible.
Those are not Morse isotopies for both knotted trivalent graphs and handlebody-knots.

\begin{figure}[htbp]
	\begin{center}
	\begin{tabular}{cc}
	\includegraphics[width=0.35\textwidth,pagebox=cropbox,clip]{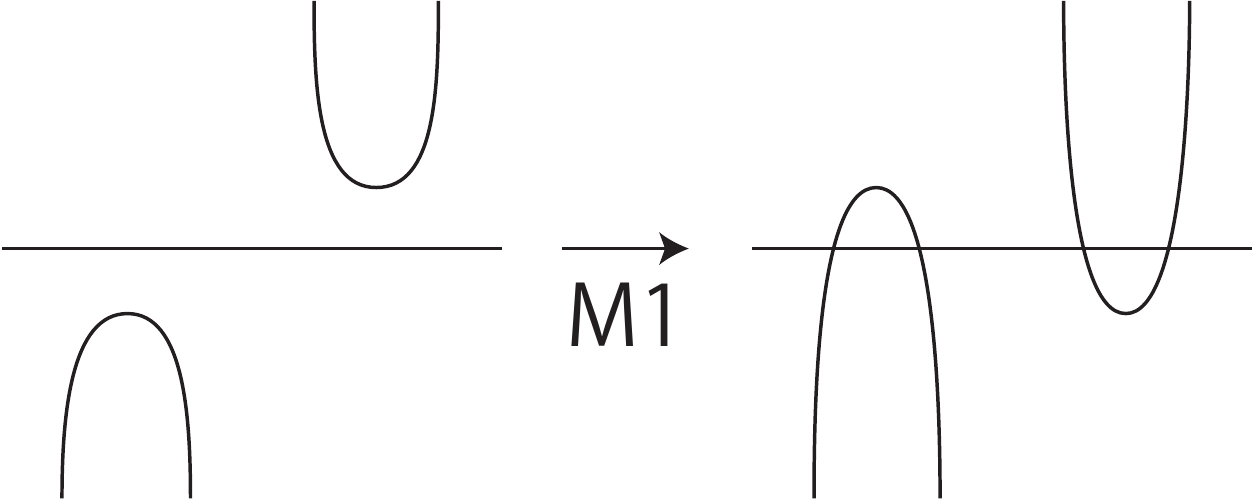}&
	\includegraphics[width=0.35\textwidth,pagebox=cropbox,clip]{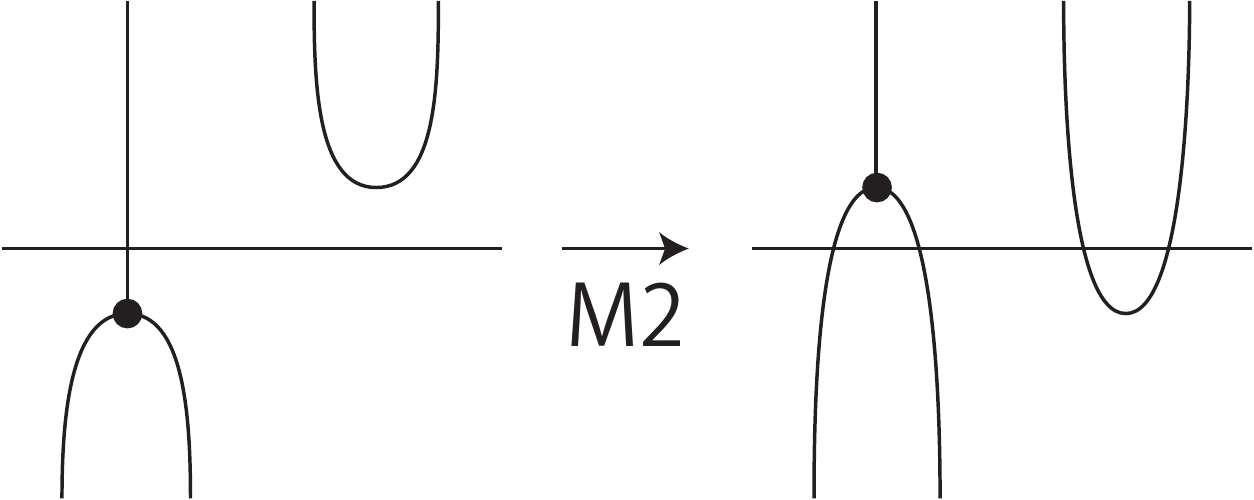}\\
	\includegraphics[width=0.35\textwidth,pagebox=cropbox,clip]{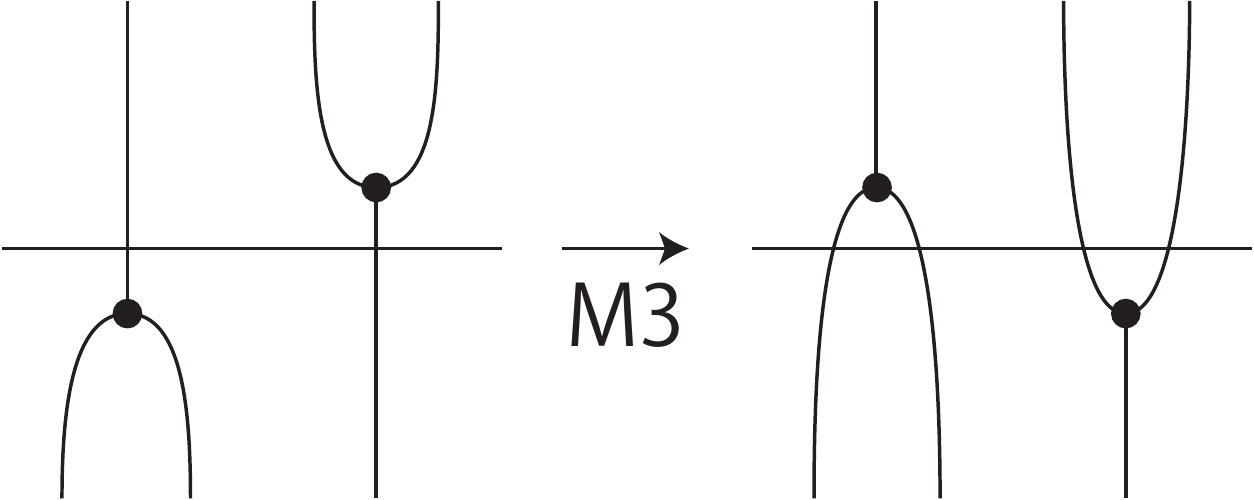}&
	\end{tabular}
	\end{center}
	\caption{$M1$, $M2$, $M3$ moves}
	\label{M123}
\end{figure}

\subsection{Correspondence between two Morse positions}

For a given Morse position $\gamma$ of a knotted trivalent graph type $\Gamma$, by taking a regular neighborhood $N(\gamma)$ of $\gamma$, we obtain a Morse position $v$ of a handlebody-knots $V$ such that $v=N(\gamma)$.
Conversely, for a given Morse position $v$ of a handlebody-knot type $V$, by taking a spine $\sigma$ of $v$, we obtain a Morse position $\gamma$ of a knotted trivalent graph type $\Gamma$ such that $\gamma=\sigma$.
In this correspondence, the definitions of equivalence for Morse positions coincides excepet for $B4$, $B5$ moves.
Hence the following holds.

\begin{proposition}\label{correspondence}
There is a correspondence between two Morse positions $v\in V$ and $\gamma\in \Gamma$.
Two Morse positions $v, v'\in V$ are equivalent if and only if corresponding $\gamma\in \Gamma$, $\gamma'\in \Gamma'$ are equivalent up to $B4$, $B5$ moves.
\end{proposition}

\section{Stable equivalence of bridge positions}

\begin{definition}
Let $\gamma\in \Gamma$ be a Morse position.
We say that $\gamma$ is a {\em bridge position} if all maximum and $\lambda$-vertices have the positive critical values, and all minimum and $Y$-vertices have the negative critical values.
\end{definition}

\begin{theorem}\label{stable1}
Let $V$ be a handlebody-knot type and $v, v'\in V$ be two bridge positions.
Then there exists a bridge position $v''\in V$ which is obtained from both $v$ and $v'$ by finite sequences of stabilizations $S1$, $S2$ and Morse isotopies.
\end{theorem}

\begin{theorem}\label{stable2}
Let $\Gamma$, $\Gamma'$ be two knotted trivalent graph types and $\gamma\in \Gamma$, $\gamma'\in \Gamma'$ be two bridge positions.
If $N(\gamma)$ and $N(\gamma')$ are equivalent as handlebody-knots, then there exist a knotted trivalent graph type $\Gamma''$ and a bridge position $\gamma''\in \Gamma''$ which is obtained from both $v$ and $v'$ by finite sequences of stabilizations $S1$, $S2$, $B4$, $B5$ moves and Morse isotopies.
\end{theorem}

By Proposition \ref{correspondence}, Theorem \ref{stable1} and \ref{stable2} are equivalent.

\subsection{Regular bridge position}

An ambient isotopy $\{f_t\}_{t\in [0,1]}$ of $S^3$ is called a {\em horizontal isotopy} (resp. {\em vertical isotopy}) if for any $t\in[0,1]$, $h\circ f_t=h\circ id_{S^3}$ (resp. $p\circ f_t=p\circ id_{S^3}$).

\begin{definition}
Let $\Gamma$ be a knotted trivalent graph type and $\gamma\in\Gamma$ be a bridge position.
We say that $\gamma$ is {\em regular} if 
\begin{enumerate}
\item $p|_{\gamma\cap B_{\pm}}$ is an injection.
\item $p|_{\gamma}$ is regular (i.e. whose multiple points are only finitely many transversal double points of the edges.).
\end{enumerate}
\end{definition}

\begin{definition}
Let $\gamma\in \Gamma$ be a bridge position.
Put $\gamma_{\pm}=\gamma\cap B_{\pm}$.
Then $\gamma_{\pm}$ is isotopic in $B_{\pm}$ fixing $\partial \gamma_{\pm}=\gamma_{\pm}\cap \partial B_{\pm}$ so that $\gamma_{\pm}$ is contained in $\partial B_{\pm}$.
Therefore, there exists a complex $\Sigma_{\pm}=\gamma_{\pm}\times [0,1]/(\partial \gamma_{\pm}\times \{t\}\sim \partial \gamma_{\pm}\times \{t'\})$ ($t,t'\in [0,1]$) embedded in $B_{\pm}$ such that $\gamma_{\pm}\times \{1\}=\gamma_{\pm}$, $\Sigma_{\pm}\cap \partial B_{\pm}=\gamma_{\pm}\times \{0\}$.
We call this complex $\Sigma_{\pm}$ is a {\em canceling complex} for $\gamma_{\pm}$.
We say that a canceling complex $\Sigma_{\pm}$ is {\em monotone} if $h|_{\rm{int}\Sigma_{\pm}}$ has no critical point, where $\rm{int}\Sigma_{\pm}=\Sigma_{\pm}-(\gamma_{\pm}\times \{0,1\})$.
Moreover, we say that a canceling complex $\Sigma_{\pm}$ is {\em vertical} if $p(\Sigma_{\pm})=\gamma_{\pm}\times \{0\}$.
\end{definition}

\begin{lemma}\label{compatible}
Let $\Gamma$ be a knotted trivalent graph type and $\gamma\in\Gamma$ be a bridge position.
Then there exists a horizontal isotopy $\{f_t\}_{t\in [0,1]}$ of $S^3$ such that $f_1(\gamma)$ is regular. 
\end{lemma}

\begin{proof}

\begin{claim}
There exists a monotone canceling complex $\Sigma_{\pm}$ for $\gamma_{\pm}$.
\end{claim}

\begin{proof}
Since $\gamma_+$ has only maxima, there exists a horizontal isotopy $\{f_t\}_{t\in [0,1]}$ of $B_+$ such that $p|_{f_1(\gamma_+)}$ is a homeomorphism.
Then there exists a vertical canceling complex $\Sigma_+$ for $f_1(\gamma_+)$.
Now we pull back $\Sigma_+$ by the inverse function $\{f_t^{-1}\}_{t\in [0,1]}$ to get a monotone canceling complex $f_1^{-1}(\Sigma_+)$ for $\gamma_+$.
\end{proof}

\begin{claim}
There exists a horizontal isotopy $\{f_t\}_{t\in [0,1]}$ taking a monotone canceling complex $\Sigma_{\pm}$ for $\gamma_{\pm}$ to a vertical canceling complex $f_1(\Sigma_{\pm})$ for $f_1(\gamma_+)$.
\end{claim}

\begin{proof}
There exists a sufficiently small positive real number $\epsilon$ such that $\Sigma_+\cap h^{-1}([0,\epsilon])$ is contained in $N(\Sigma_+\cap h^{-1}(0))\times [0,\epsilon]$.
Since $\Sigma_+$ is monotone, there exists a horizontal isotopy $\{g_t\}_{t\in [0,1]}$ of $h^{-1}([0,\epsilon])$ such that $g_1(\Sigma_+\cap h^{-1}([0,\epsilon]))$ is vertical.
To define a horizontal isotopy of the rest $h^{-1}((\epsilon,\infty))$, 
let $g_t|_{h^{-1}(x)} = g_t|_{h^{-1}(\epsilon)}$, for any $x\in (\epsilon,\infty)$ and $t\in[0,1]$.
Repeating this procedure, we will obtain a horizontal isotopy $\{f_t\}_{t\in [0,1]}$ as desired.
\end{proof}

The above two claims show Lemma \ref{compatible}.
\end{proof}

Let $\gamma\in \Gamma$ be a regular bridge position.
Then, we have a projection $p(\gamma)$.
For each double point $x$ of $p(\gamma)$, there are two points $p^{-1}(x)$ one of which lies $\gamma_+$ and another lies $\gamma_-$.
We obtain a diagram $\widetilde{p(\gamma)}$ by adding an over/under information on each double point $x$ depending on whether each point of $p^{-1}(x)$ lies $\gamma_+$ or $\gamma_-$.

To prove Theorem \ref{stable2}, we use the following theorem.

\begin{theorem}[{\cite[Corollary 2]{I}}]\label{R}
Let $D_1$ and $D_2$ be diagrams of spatial trivalent graphs $L_1$ and $L_2$, respectively. Then $L_1$ and $L_2$ are neighborhood equivalent if and only if $D_1$ and $D_2$
are related by a finite sequence of the moves $R1-6$ depicted in Figure \ref{R1-6}.
\end{theorem}

\begin{figure}[htbp]
	\begin{center}
	\includegraphics[width=0.8\textwidth,pagebox=cropbox,clip]{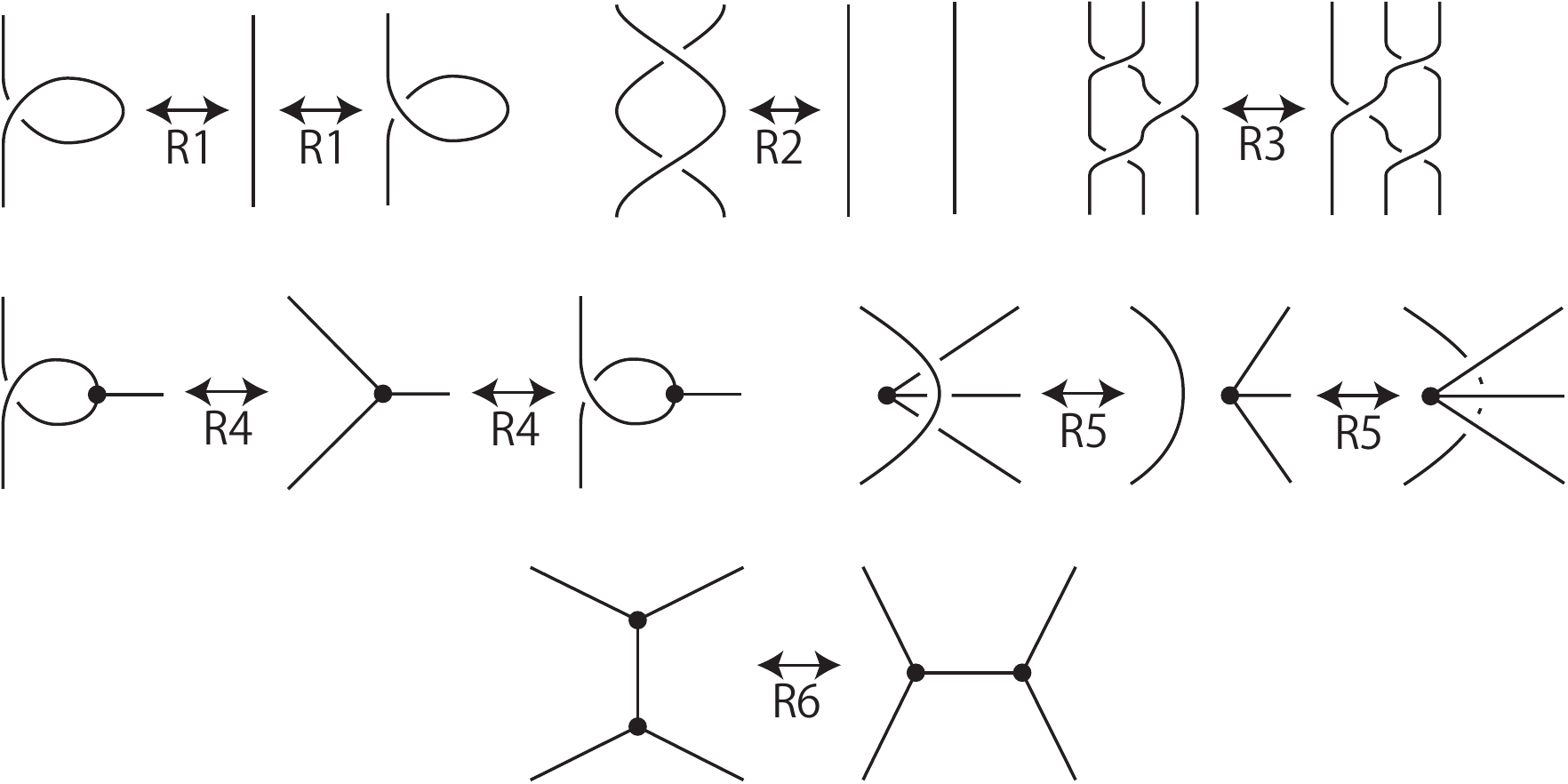}
	\end{center}
	\caption{Moves for diagrams of knotted trivalent graphs}
	\label{R1-6}
\end{figure}

\subsection{Stable equivalence of regular bridge positions}

By Lemma \ref{compatible}, we may assume that $\gamma$ and $\gamma'$ are regular by horizontal isotopies.
To prove Theorem \ref{stable1} or equivalently Theorem \ref{stable2}, it suffices to show the following theorem.

\begin{theorem}\label{stable3}
Let $\gamma$, $\gamma'$ be two regular bridge positions.
Suppose that two diagrams $\widetilde{p(\gamma)}$ and $\widetilde{p(\gamma')}$ are related by a finite sequence of the moves $R1-6$.
Then  there exist two sequences of regular bridge positions $\gamma=\gamma_1, \gamma_2, \ldots, \gamma_n=\gamma''$ and $\gamma'=\gamma'_1, \gamma'_2, \ldots, \gamma'_m=\gamma''$, where $\gamma_{i+1}$ $($resp. $\gamma'_{j+1}$$)$ is obtained from $\gamma_i$ $($resp. $\gamma'_j$$)$ by one from among stabilizations $S1$, $S2$, $B4$, $B5$ moves and a Morse isotopy $(i=1,\ldots, n-1; j=1,\ldots, m-1)$.
\end{theorem}


\begin{definition}
For a regular bridge position $\gamma$, both stabilizations $S1$ and $S2$ can be performed by vertical isotopies.
We call such stabilizations {\em vertical stabilizations}.
\begin{figure}[htbp]
	\begin{center}
	\includegraphics[width=0.7\textwidth,pagebox=cropbox,clip]{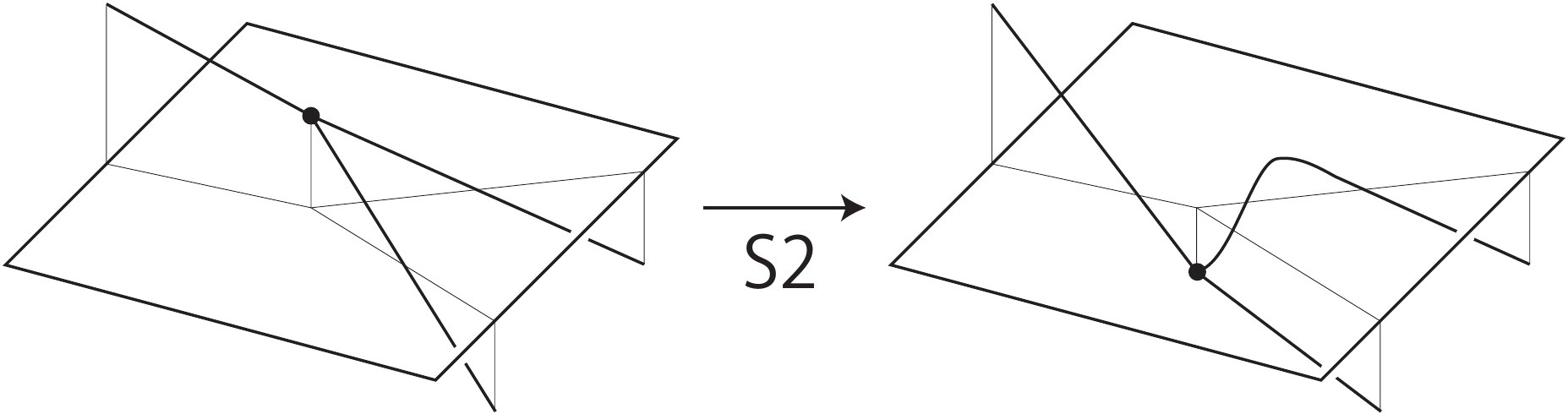}
	\end{center}
	\caption{A vertical stabilization $S2$ with a vertical canceling complex}
	\label{verticalS2}
\end{figure}
\end{definition}

\begin{lemma}\label{intersects}
Let $\gamma$ be a regular bridge position.
By a finite sequence of vertical stabilizations $S1$, $S2$ and vertical isotopies, we may assume that for any subarc $\alpha$ of $\gamma_{+}$ $($resp. $\gamma_{-}$$)$ except for all crossings of $\widetilde{p(\gamma)}$, $\alpha$ is a component of $\gamma_{-}$ $($resp. $\gamma_{+}$$)$.
\end{lemma}

\begin{proof}
Without loss of generality, let $\alpha$ be a subarc of $\gamma_+$ except for all crossings of $\widetilde{p(\gamma)}$.
Let $e$ be an edge of $\gamma$ containing $\alpha$.

First suppose that $e$ intersects the bridge sphere $S^2$.
As shown in Figure \ref{deform0}, we perform a vertical stabilization $S1$ at the point of $e\cap S^2$, and then two small upper bridge and lower bridge are created.
This lower bridge can be moved into the below of $\alpha$ by a vertical isotopy.
Hence $\alpha$ is a component of $\gamma_{-}$.

\begin{figure}[htbp]
	\begin{center}
	\includegraphics[width=0.8\textwidth,pagebox=cropbox,clip]{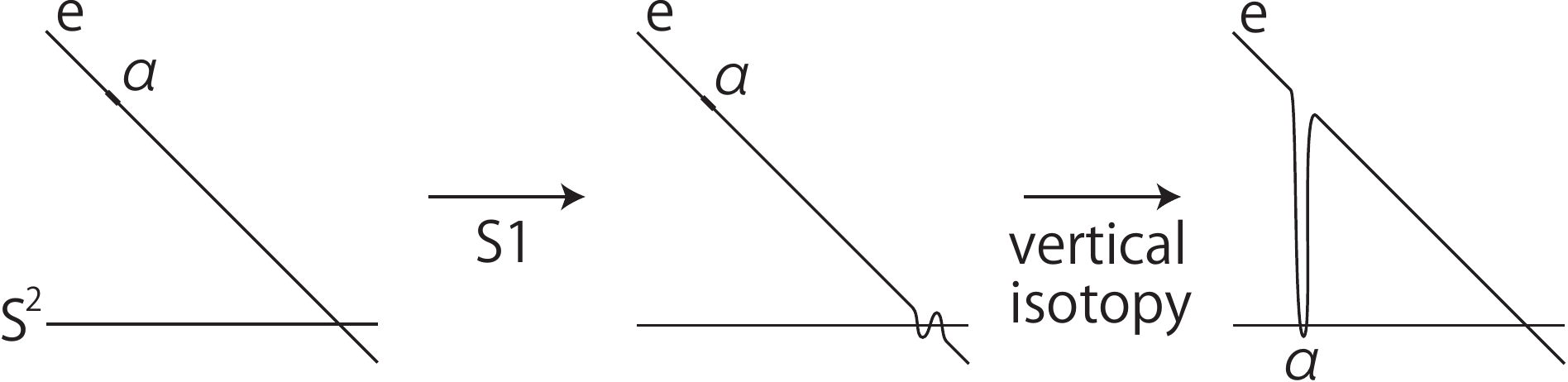}
	\end{center}
	\caption{A deformation of $\alpha$ by a stabilization $S1$ and a vertical isotopy}
	\label{deform0}
\end{figure}

Next suppose that the edge $e$ does not intersect the bridge sphere $S^2$.
In the component of $\gamma_+$ containing $e$, let $x$ be the lowest $\lambda$-vertex.
Let $e_1, e_2, e_3$ be the edges incident to $x$ such that $e_1$ lies above $x$ and $e_2, e_3$ lie below $x$, Since $x$ is the lowest $\lambda$-vertex, both of $e_2$ and $e_3$ intersect the bridge sphere $S^2$.

Put $e_2'\subset e_2$ and $e_3'\subset e_3$ as portions of the component of $\gamma_+$ containing $e$.
First we perform a vertical stabilization $S1$ at the endpoint of $e_2'$ on $S^2$.
Next by a vertical isotopy, we have $p(e_2')$ does not intersect $p(\gamma_-)$, that is, $p(e_2')$ has no crossing of $\widetilde{p(\gamma)}$.
See Figure \ref{deform1}.
The same applies to $e_3'$.

\begin{figure}[htbp]
	\begin{center}
	\includegraphics[width=0.8\textwidth,pagebox=cropbox,clip]{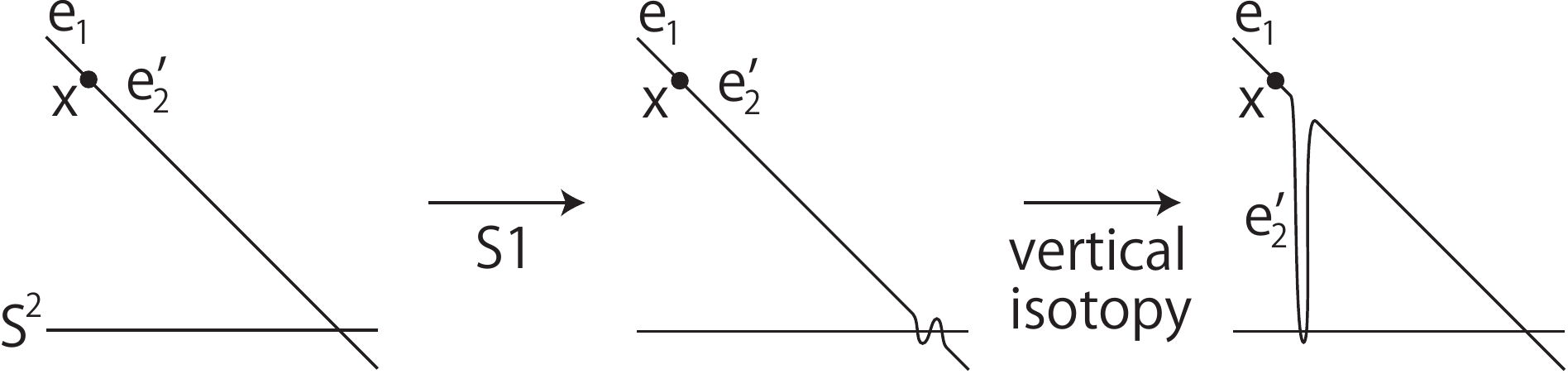}
	\end{center}
	\caption{A deformation of $e_2'$ by a stabilization $S1$ and a vertical isotopy}
	\label{deform1}
\end{figure}

Since both of $p(e_2')$ and $p(e_3')$ have no crossing of $\widetilde{p(\gamma)}$, we can perform a vertical stabilization $S2$ around $x$.
After a vertical stabilization $S2$ around $x$, $e_1$ intersects the bridge sphere $S^2$.
See Figure \ref{verticalS2}.

By repeating these operations, the edge $e$ eventually intersects the bridge sphere $S^2$.
Then the argument follows the previous case.
\end{proof}

By Theorem \ref{R}, two diagrams $\widetilde{p(\gamma)}$ and $\widetilde{p(\gamma')}$ are related by a finite sequence of the moves $R1-6$.

\begin{lemma}\label{Reidemeister}
A finite sequence of stabilizations $S1$, $S2$, $B4$, $B5$ moves and Morse isotopies can be substituted for the finite sequence of the moves $R1-6$ and planar isotopies.
\end{lemma}

\begin{proof}
$(R1)$ Figure \ref{R1} shows that Morse isotopies can be substituted for one direction of $R1$.

\begin{figure}[htbp]
	\begin{center}
	\includegraphics[width=0.3\textwidth,pagebox=cropbox,clip]{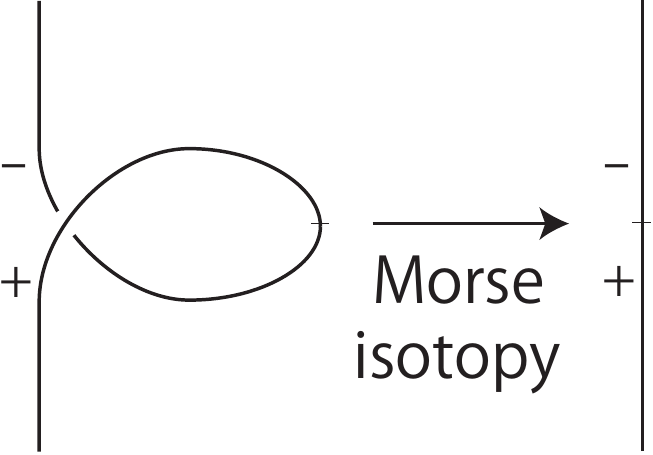}
	\end{center}
	\caption{Morse isotopies can be substituted for $R1$.}
	\label{R1}
\end{figure}

For another direction of $R1$, by Lemma \ref{intersects}, we may assume that any subarc in a given edge intersects the bridge sphere $S^2$.
Then, by the converse deformation of Figure \ref{R1}, Morse isotopies are substituted for another direction of $R1$.

$(R2)$ Following a deformation as shown in Figure \ref{deform0}, we may assume that any subarc does not intersect the bridge sphere $S^2$.
Then Figure \ref{R2} shows that Morse isotopies can be substituted for one direction of $R2$.

\begin{figure}[htbp]
	\begin{center}
	\includegraphics[width=0.4\textwidth,pagebox=cropbox,clip]{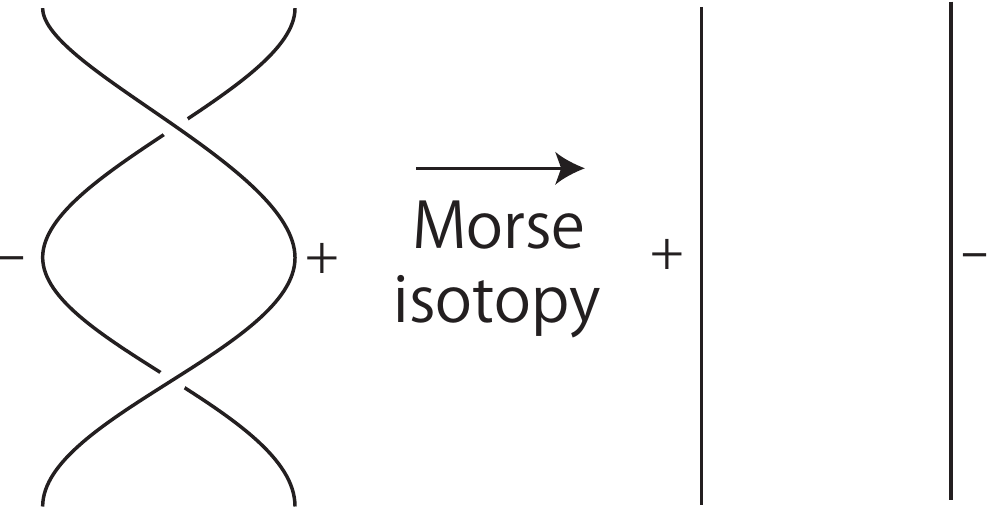}
	\end{center}
	\caption{Morse isotopies can be substituted for $R2$.}
	\label{R2}
\end{figure}

For another direction of $R2$, by Lemma \ref{intersects}, we may assume that any two subarcs lie in $B_+$ and $B_-$ respectively. 
Then, by the converse deformation of Figure \ref{R2}, Morse isotopies are substituted for another direction of $R2$.

$(R3)$ Figure \ref{R3} shows that a stabilization $S1$ and Morse isotopies can be substituted for one direction of $R3$.
The same applies to another direction of $R3$.

\begin{figure}[htbp]
	\begin{center}
	\includegraphics[width=0.8\textwidth,pagebox=cropbox,clip]{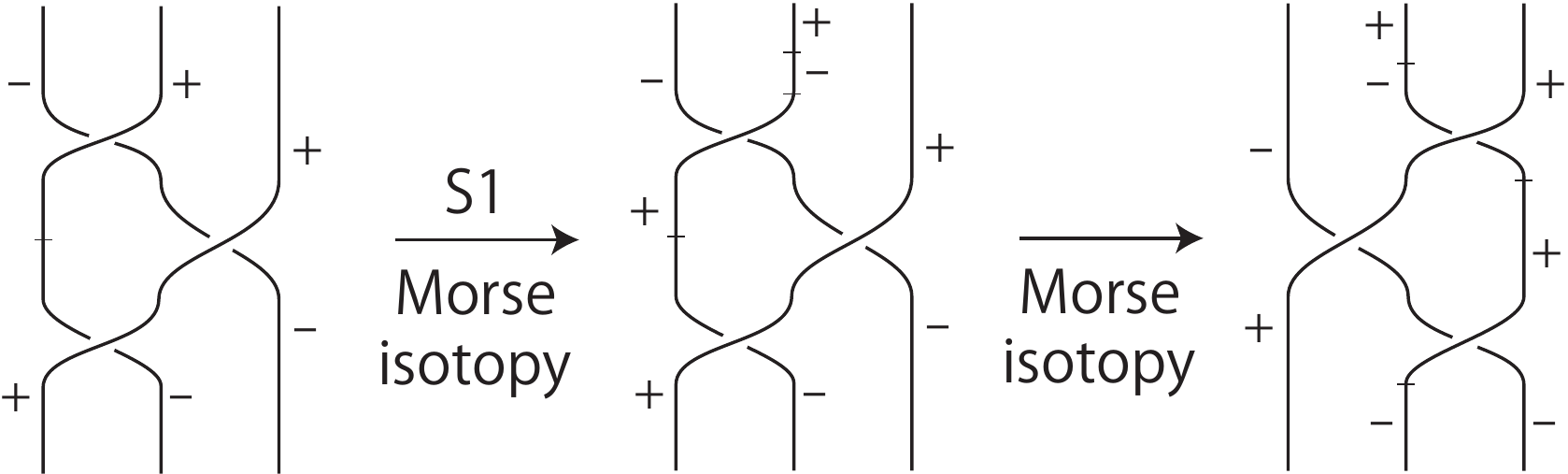}
	\end{center}
	\caption{A stabilization $S1$ and Morse isotopies can be substituted for $R3$.}
	\label{R3}
\end{figure}

$(R4)$ We consider one direction of $R4$ as shown in Figure \ref{R4}, where around the vertex $x$, two edges $e_1$ and $e_2$ contain over and under crossings respectively, $e_3$ contains no crossing.
Following a deformation as shown in Figure \ref{deform0}, we may assume that around the vertex $x$, only $e_2$ intersects the bridge sphere $S^2$.

\begin{figure}[htbp]
	\begin{center}
	\includegraphics[width=0.55\textwidth,pagebox=cropbox,clip]{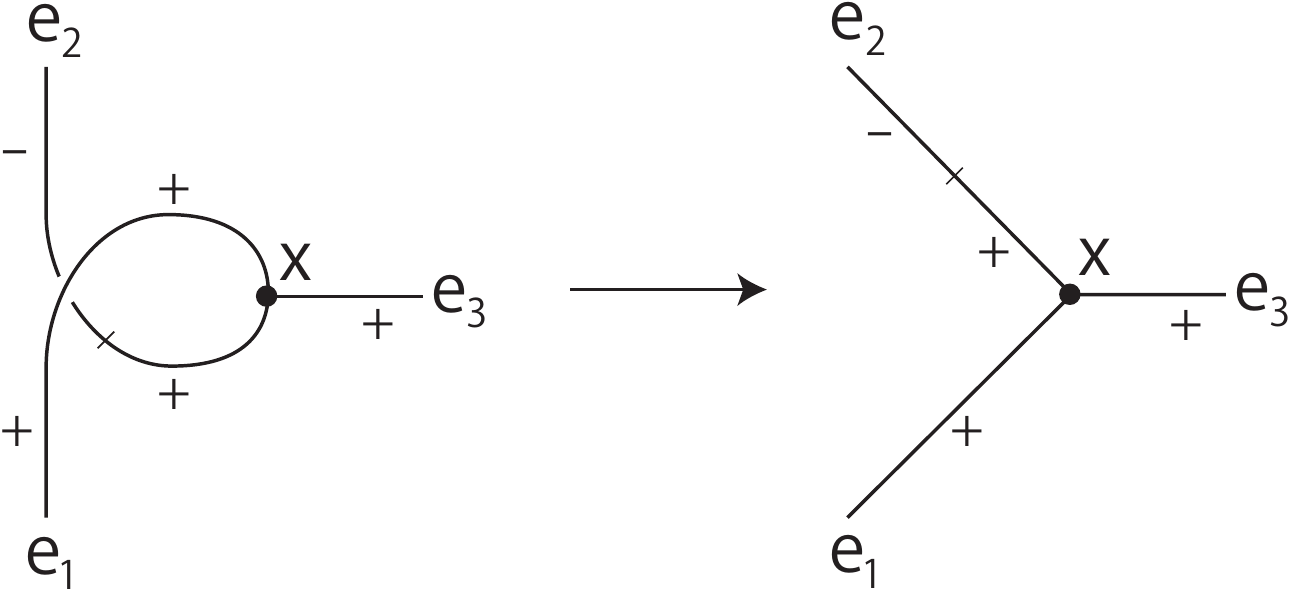}
	\end{center}
	\caption{One direction of $R4$}
	\label{R4}
\end{figure}

Since $x$ is a $\lambda$-vertex, there are three cases for consideration.
\begin{description}
\item[Case 1] One end of $e_1$ lies above $x$ and two ends of $e_2, e_3$ lie below $x$.
\item[Case 2] One end of $e_2$ lies above $x$ and two ends of $e_1, e_3$ lie below $x$.
\item[Case 3] One end of $e_3$ lies above $x$ and two ends of $e_1, e_2$ lie below $x$.
\end{description}

In Case 1, around the vertex $x$, $e_1$ lies entirely above $e_2$.
Hence one direction of $R4$ can be realized by a Morse isotopy.

In Case 3, around the vertex $x$, we may assume that $e_1$ lies entirely above $e_2$ by a Morse isotopy.
Then one direction of $R4$ can be realized by a Morse isotopy.

In Case 2, $e_2$ must have a maximal point, say $y$.
By a $B4$ move, we slide $e_3$ beyond $y$.
Then the situation is similar to Case 1, and one direction of $R4$ can be realized by a Morse isotopy.
See Figure \ref{R4B4}.

\begin{figure}[htbp]
	\begin{center}
	\includegraphics[width=0.8\textwidth,pagebox=cropbox,clip]{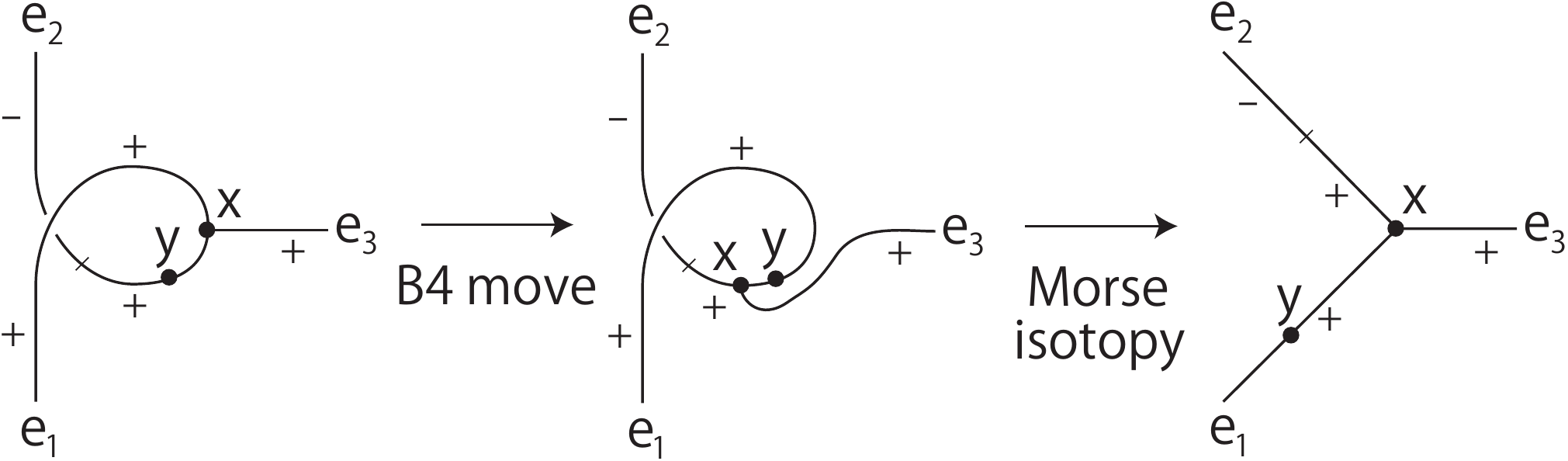}
	\end{center}
	\caption{One direction of $R4$}
	\label{R4B4}
\end{figure}

Next we consider the reverse direction of $R4$ as shown in Figure \ref{R4}.
Since $x$ is a $\lambda$-vertex, there are three cases for consideration.
\begin{description}
\item[Case 1] One end of $e_1$ lies above $x$ and two ends of $e_2, e_3$ lie below $x$.
\item[Case 2] One end of $e_2$ lies above $x$ and two ends of $e_1, e_3$ lie below $x$.
\item[Case 3] One end of $e_3$ lies above $x$ and two ends of $e_1, e_2$ lie below $x$.
\end{description}

In Case 1, around the vertex $x$, $e_1$ lies entirely above $e_2$.
Hence the reverse direction of $R4$ can be realized by a Morse isotopy.

In Case 3, around the vertex $x$, we may assume that $e_1$ lies entirely above $e_2$ except for $x$ by a Morse isotopy.
Then the reverse direction of $R4$ can be realized by a Morse isotopy.

In Case 2, by Lemma \ref{intersects}, we may assume that around the vertex $x$, $e_1$ and $e_3$ intersect the bridge sphere $S^2$.
By a vertical stabilization $S2$, $x$ becomes a $Y$-vertex, $e_1$ intersects $S^2$ in two points, $e_2$ intersects $S^2$ in one point and $e_3$ does not intersect $S^2$.
Then the reverse direction of $R4$ can be realized by a Morse isotopy.
See Figure \ref{R4S2}.

\begin{figure}[htbp]
	\begin{center}
	\includegraphics[width=0.8\textwidth,pagebox=cropbox,clip]{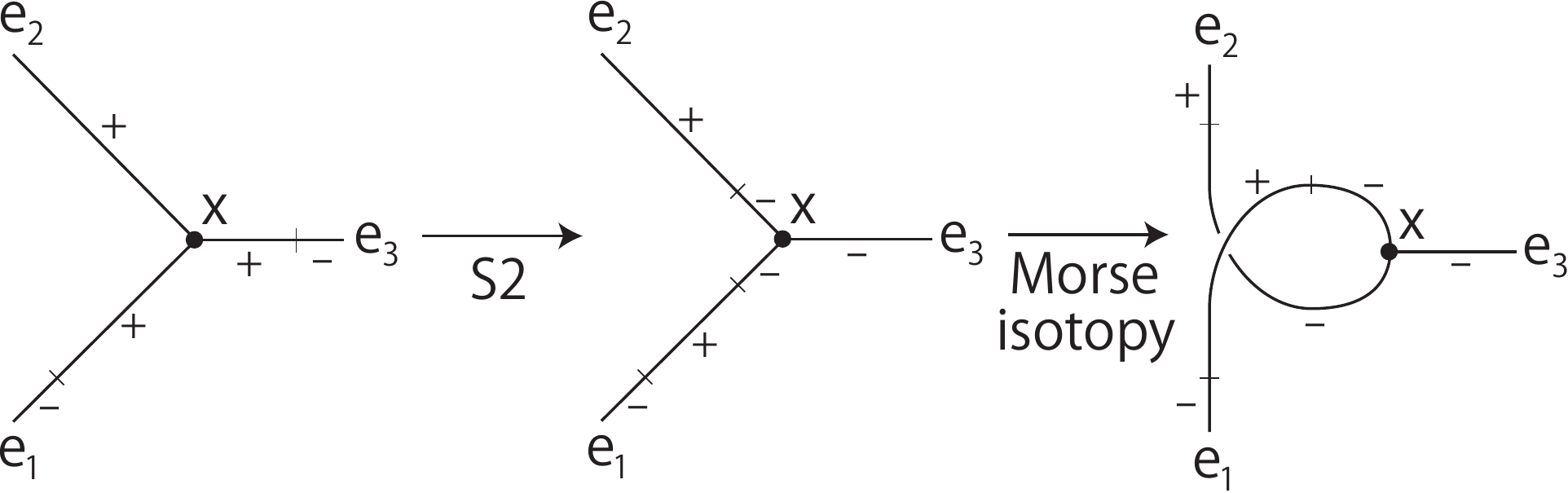}
	\end{center}
	\caption{The reverse direction of $R4$}
	\label{R4S2}
\end{figure}

$(R5)$ Following a deformation as shown in Figure \ref{deform0}, we may assume that $e_4$ does not intersect the bridge sphere $S^2$.
Then by a vertical isotopy, we may assume that $e_4$ lies above $e_1, e_2, e_3$, and one direction of $R5$ can be realized by a Morse isotopy as shown in Figure \ref{R5}.

\begin{figure}[htbp]
	\begin{center}
	\includegraphics[width=0.45\textwidth,pagebox=cropbox,clip]{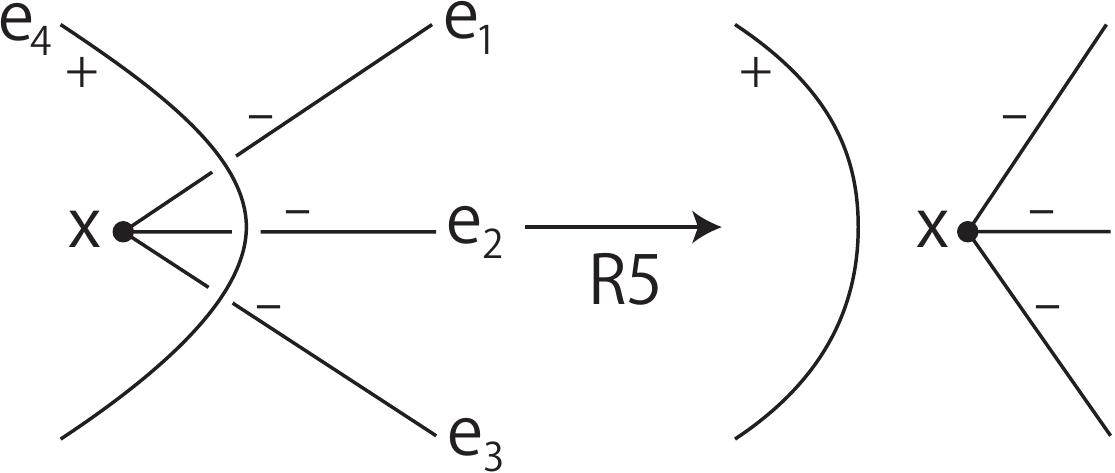}
	\end{center}
	\caption{One direction of $R5$.}
	\label{R5}
\end{figure}

For another direction of $R5$, by Lemma \ref{intersects}, we may assume that $e_4$ lies in $B_+$ and lies above $e_1, e_2, e_3$ by a vertical isotopy.
Then, the coverse direction of $R5$ can be realized by a Morse isotopy.

$(R6)$ We label each edges and vertices as shown in Figure \ref{R6}.
We note that $p(e_1)$ does not contain any crossing of $\widetilde{p(\gamma)}$.
First we consider the case that $e_1$ does not intersect the bridge sphere $S^2$, and $x_1, x_2$ lies in $B_+$.
If $e_1$ has a maximal point, then by a $B4$ move, we may assume that $e_1$ does not have a maximal point.
Then one direction of $R6$ can be realized by a $B5$ move.

\begin{figure}[htbp]
	\begin{center}
	\includegraphics[width=0.55\textwidth,pagebox=cropbox,clip]{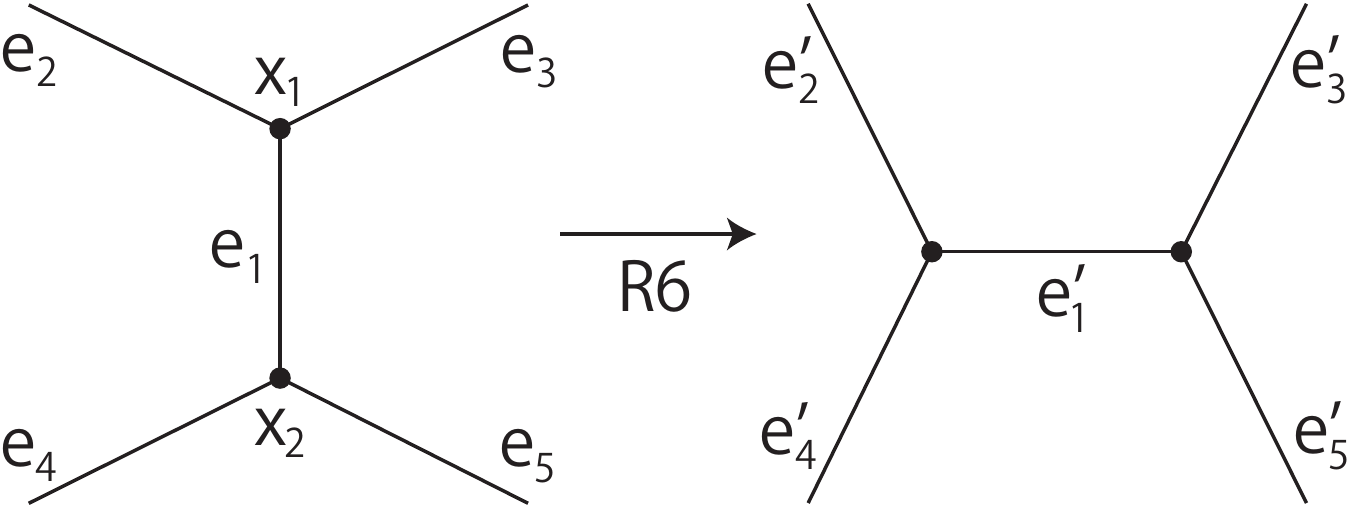}
	\end{center}
	\caption{One direction of $R6$}
	\label{R6}
\end{figure}

Next we consider the case that $e_1$ intersects the bridge sphere $S^2$.
By repeating a stabilization $S2$ and a $B4$ move alternatively, $e_1$ eventually does not intersect $S^2$ as follows.
Then as the previous case, one direction of $R6$ can be realized by a $B5$ move.

To show that the intersection of $e_1$ and $S^2$ can be reduced by repeating a stabilization $S2$ and a $B4$ move alternatively, we may assume without loss of generality that $x_2$ lies in $B_+$.
First we consider the case that the portion of $e_1$ which incidents to $x_2$ contains a maximal point $y$.
Then by a $B4$ move, we slide $e_4$ so that the portion of $e_1$ does not contain $y$, but instead, $e_5$ contains a maximal point.
See Figure \ref{R6B4}.

\begin{figure}[htbp]
	\begin{center}
	\includegraphics[width=0.55\textwidth,pagebox=cropbox,clip]{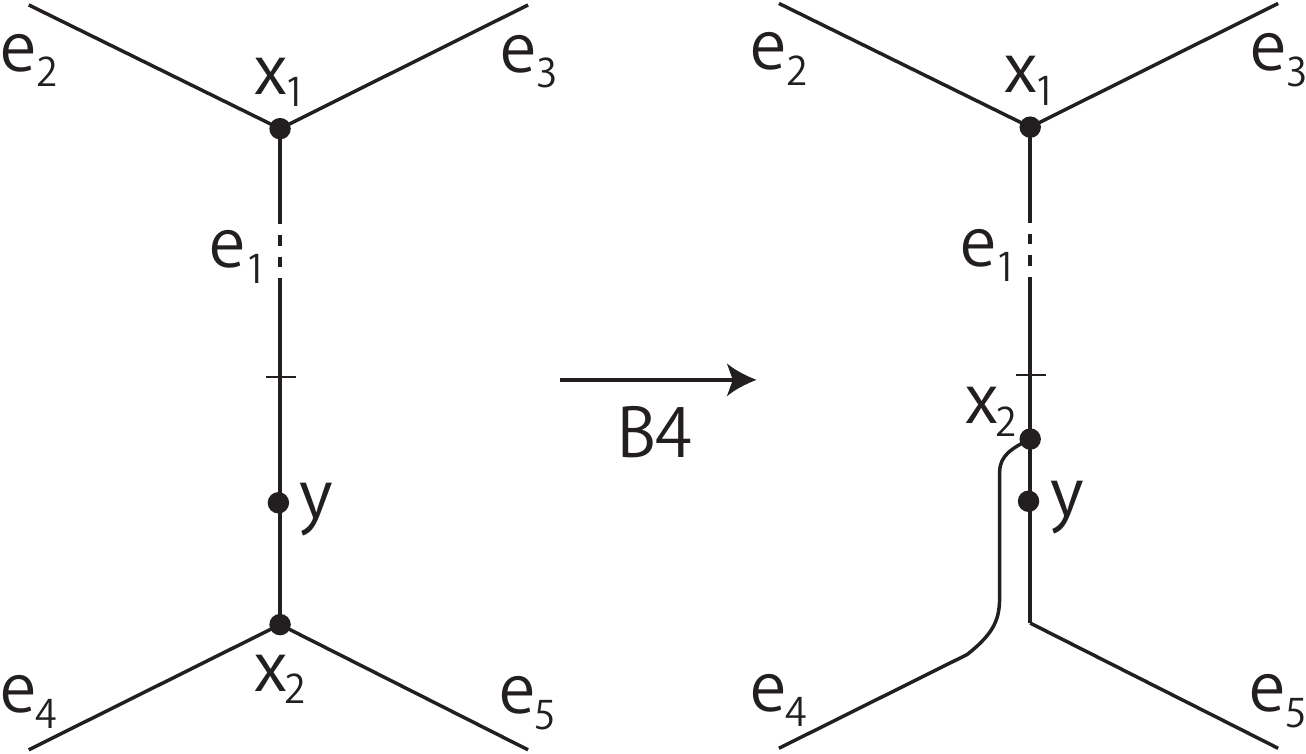}
	\end{center}
	\caption{A $B4$ move}
	\label{R6B4}
\end{figure}

Next we consider the case that the portion of $e_1$ which incidents to $x_2$ does not contain a maximal point.
We may assume without loss of generality that $e_5$ contains a maximal point $y$ and following a deformation as shown in Figure \ref{deform0}, $e_4$ intersects the bridge sphere $S^2$ around $x_2$.
By a vertical stabilization $S2$, the number of the intersection of $e_1$ and $S^2$ can be reduced by one.
See Figure \ref{R6S2}.

\begin{figure}[htbp]
	\begin{center}
	\includegraphics[width=0.55\textwidth,pagebox=cropbox,clip]{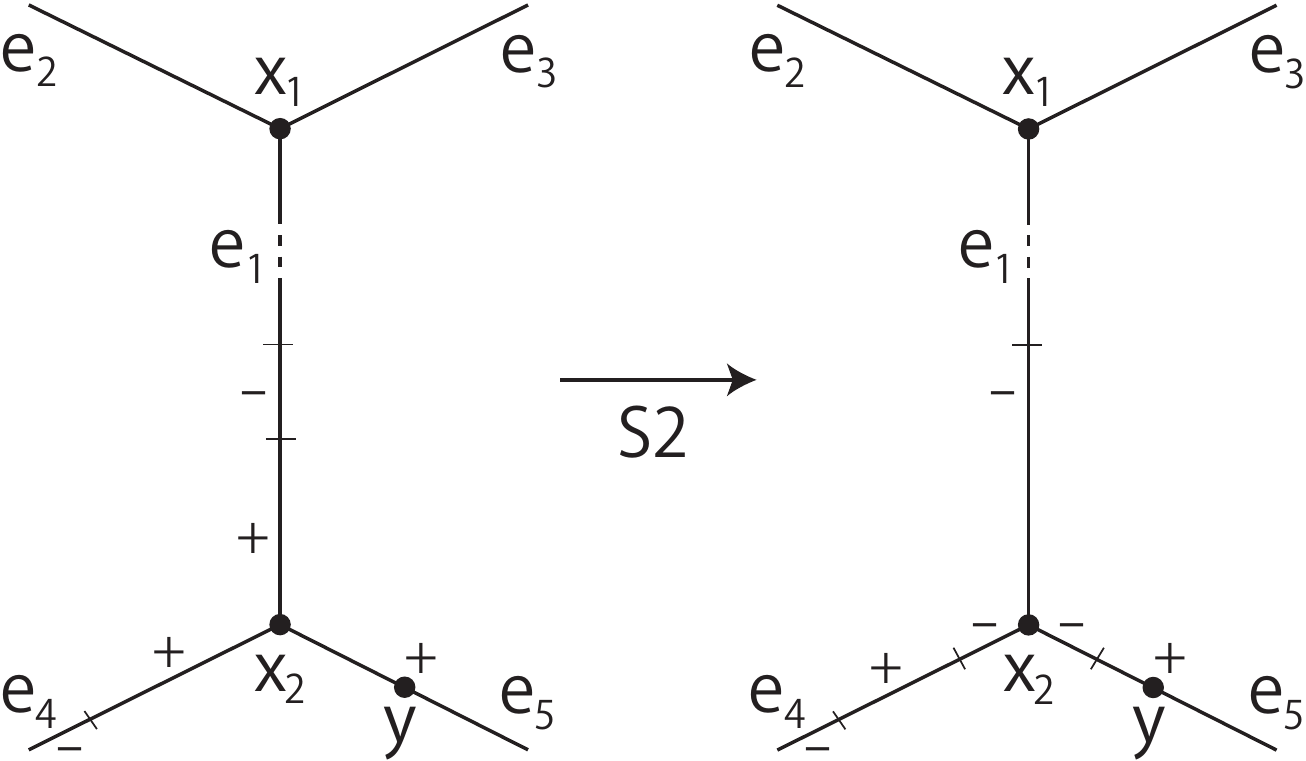}
	\end{center}
	\caption{A stabilization $S2$}
	\label{R6S2}
\end{figure}

This completes the proof of Lemma \ref{Reidemeister}.
\end{proof}

By Lemma \ref{Reidemeister}, $\gamma$ and $\gamma'$ are same as diagrams.
However, $\gamma$ and $\gamma'$ are not necessarily same as regular bridge positions.

\begin{lemma}
Let $\gamma$ and $\gamma'$ be regular bridge positions such that $\widetilde{p(\gamma)}=\widetilde{p(\gamma')}$.
Then by a finite sequence of vertical stabilizations $S1$, $S2$ and $B4$ moves, we have $\gamma=\gamma'$.
\end{lemma}

\begin{proof}
By Lemma \ref{intersects}, we may assume that each edge of $\gamma$ and $\gamma'$ intersects the bridge sphere $S^2$.
Moreover, for simplicity, we may assume that any subarc between two crossings or between a crossing and a vertex of $\gamma$ and $\gamma'$ intersects $S^2$.
On each crossing, by vertical isotopies, $\gamma$ and $\gamma'$ coincide.
On each corresponding vertices $x$ and $x'$ , we may assume that both of $x$ and $x'$ lie in $B_+$ or $B_-$ by a stabilization $S2$.
Then by a $B4$ move, $\gamma$ and $\gamma'$ coincide around the vertices.
Finally, by stabilizations $S1$, we may assume that the intersection number of $S^2$ and each subarc between two crossings or between a crossing and a vertex of $\gamma$ and $\gamma'$ coincide.
Hence by a vertical isotopy, we have $\gamma=\gamma'$.
\end{proof}

\subsection{Relation among Morse positions}

\begin{proposition}[cf. \cite{D}, {\cite[Theorem 5.1 (1)]{II}}]\label{related}
Let $V$ be a handlebody-knot type and $v, v'\in V$ be two Morse positions.
Then two equivalence classes $[v]$ and $[v']$ are related by a finite sequence of stabilizations $S1$, $S2$ and their destabilizations and $M1$, $M2$, $M3$ moves.
\end{proposition}

\begin{proof}
First by applying $S3$, $M1$, $M2$, $M3$ to $v, v'$, we obtain two bridge positions $v, v'$.
We remark that a stabilization $S3$ is obtained by a sequence of $S2$, $M2$, $B5$ (Morse isotopy) and a destabilizaiton of $S2$.
See Figure \ref{S33}.

\begin{figure}[htbp]
	\begin{center}
	\includegraphics[width=0.6\textwidth,pagebox=cropbox,clip]{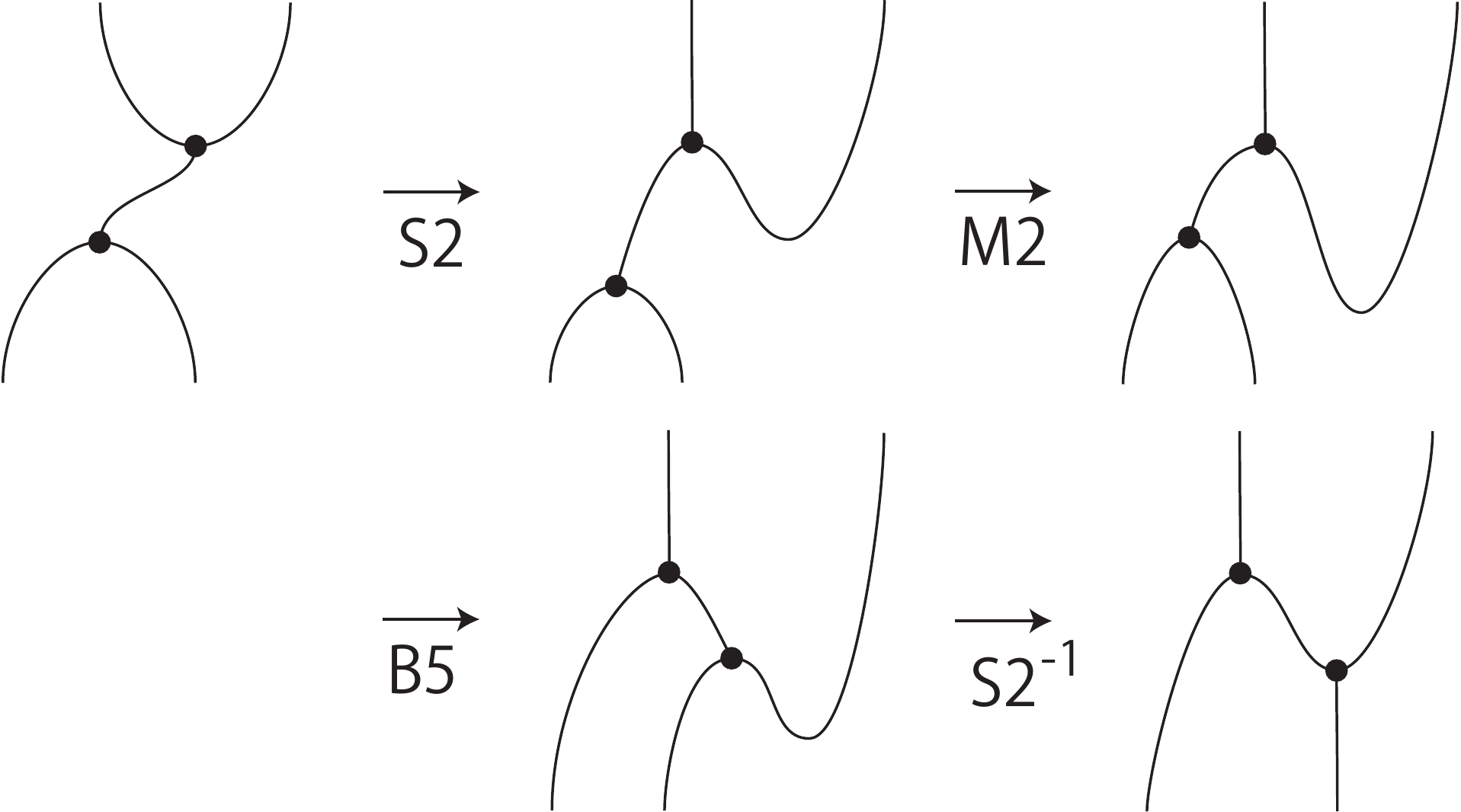}
	\end{center}
	\caption{A stabilization $S3$ is obtained by $S2$, $M2$, $B5$, $S2^{-1}$}
	\label{S33}
\end{figure}

Then by Theorem \ref{stable1}, two bridge positions $v, v'$ are related by $S1$, $S2$ and their destabilizations.
\end{proof}

\bibliographystyle{amsplain}

\begin{thebibliography}{10}

\bibitem{B} J. S. Birman,  {\em On the stable equivalence of plat representations of knots and links}, Canad. J. Math. \textbf{28} (1976) 264--290.

\bibitem{C} R. Craggs, {\em A new proof of the Reidemeister--Singer theorem on stable equivalence of Heegaard splittings}, Proc. Amer. Math. Soc. {\bf 57} (1976), 143--147.

\bibitem{D} Z. Dancso, {\em On the Kontsevich integral for knotted trivalent graphs}, Algebr. Geom. Topol. {\bf 10} (2010), 1317--1365.


\bibitem{GST} H. Goda, M. Scharlemann, A. Thompson, {\em Leveling an unknotting tunnel}, Geom. Topol. {\bf 4} (2000), 243--275.

\bibitem{H} C. Hayashi, {\em Stable equivalence of Heegaard splittings of 1-submanifolds in 3-manifolds}, Kobe J. Math. {\bf 15} (1998), 147--156.

\bibitem{II} K. Ishihara, A. Ishii, {\em An operator invariant for handlebody-knots}, Fund. Math. {\bf 217} (2012), 233--247.

\bibitem{I} A. Ishii, {\em Moves and invariants for knotted handlebodies}, Algebr. Geom. Topol. {\bf 8} (2008), 1403--1418.

\bibitem{J} J. Johnson, {\em Stable functions and common stabilizations of Heegaard splittings}, Trans. Amer. Math. Soc. {\bf 361} (2009), 3747--3765.

\bibitem{L} F. Lei, {\em On stability of Heegaard splittings}, Math. Proc. Camb. Phil. Soc. {\bf 129} (2000), 55--57.

\bibitem{La} F. Laudenbach, {\em A proof of Reidemeister--Singer's theorem by Cerf's methods}, Ann. Fac. Sci. Toulouse Math. (6) {\bf 23} (2014), 197--221. 



\bibitem{Re} K. Reidemeister, {\em Zur dreidimensionalen Topologie}, Abh. Math. Sem. Univ. Hamburg {\bf 11}
(1933), 189--194.

\bibitem{S} M. Scharlemann, {\em Thin Position in the Theory of Classical Knots}, in Handbook of Knot Theory, Elsevier Science (2005), 429--459.

\bibitem{ST} M. Scharlemann, A. Thompson, {\em Thin position and Heegaard splittings of the 3-sphere}, J. Diff. Geom. {\bf 39} (1994), 343--357.

\bibitem{Sie} L. Siebenmann, {\em Les bissections expliquent le th\'{e}or\`{e}me de Reidemeister--Singer : Un retour aux sources}, Ann. Fac. Sci. Toulouse Math. (6) {\bf 24} (2015), 1025--1056.

\bibitem{Si} J. Singer, {\em Three-dimensional manifolds and their Heegaard diagrams}, Trans. Amer. Math.
Soc. {\bf 35} (1933), 88--111.


\bibitem{Z} A. Zupan, {\em Bridge and pants complexities of knots}, J. London Math. Soc. (2) {\bf 87} (2013), 43--68.


\end{thebibliography}

\bigskip
\noindent{\bf Acknowledgements.}
The author would like to thank Kazuto Takao and Atsushi Ishii for useful comments.

\end{document}